\documentclass[a4paper,12pt,reqno]{amsart}

\usepackage{amsmath}
\usepackage{amssymb}
\usepackage{amsfonts}
\usepackage{mathrsfs}
\usepackage{a4wide}
\usepackage{color}

\makeatletter
\@addtoreset{equation}{section}
\makeatother

\newtheorem{statement}{}[section]
\newtheorem{theo}[statement]{Theorem}

\newtheorem{proposition}[statement]{Proposition}
\newtheorem{defi}[statement]{Definition}

\newtheorem{lem}[statement]{Lemma}

\newtheorem{example}[statement]{Example}

\newcommand{\Hol}{\mathrm{Hol}}
\newcommand{\DD}{\mathbb{D}}

\newcommand{\N}{\mathbb{N}}
\newcommand{\R}{\mathbb{R}}

\newcommand{\Lip}{\mathrm{Lip }}
\newcommand{\kH}{\mathrm{H^\infty  }}
\newcommand{\hH}{\mathrm{H^2 }}

\newcommand{\IL}{\mathrm{L^2}}
\newcommand{\Card}{\mathrm{Card}}

\newcommand\CC{\mathbb{C}}
\newcommand\ZZ{\mathbb{Z}}
\newcommand\TT{\mathbb{T}}

\newcommand\LB{{B_\Lambda^{+}}}
\usepackage{a4wide}

\DeclareMathOperator{\dist}{dist}

\begin{document}
\title{Cyclic vectors in Korenblum type spaces}

\author{Abdelouahab Hanine}
\date{ \today}
\address{ Laboratoire Analyse et Applications -- URAC/03. D\'{e}partement de Math\'emati\-ques,  Universit\'e Mohammed V, Rabat--Agdal--B.P. 1014 Rabat, Morocco}
\address{LATP\\ Aix-Marseille Universit\'e\\39 rue F. Joliot-Curie\\ 13453 Marseille
\\France}
\email{hanine@cmi.univ-mrs.fr, abhanine@gmail.com}

\thanks{This work was partially supported by ``Hassan II Academy of Science and Technology" and  by 
the Egide Volubilis program.}

\begin{abstract}
In this paper we use the technique of premeasures, introduced by Korenblum in the 1970-s, to give a characterization of cyclic functions in the Korenblum type spaces $\mathcal{A}_{\Lambda}^{-\infty }$. In particular, we give a positive answer to 
a conjecture by Deninger \cite[Conjecture 42]{CD}.
\end{abstract}
\maketitle

\section{Introduction}
Let $\mathbb{D}$ be the open unit disk in the complex plan $\CC$. Suppose that $X$ is a topological vector space of analytic functions on $\DD$, with the property that $zf\in X$ whenever $f\in X$. Multiplication by $z$ is thus an operator on $X$, and if $X$ is a Banach space, then it is automatically a bounded operator on space $X$. A closed subspace $M\subset X$ (Banach space) is said to be invariant (or $z$-invariant) provided that $zM\subset M$.  For a function $f\in X$, the closed linear span in $X$ of all polynomial multiples of $f$ is an $z$-invariant subspace denoted by $[f]_{X}$; it is also the smallest closed $z$-invariant subspace of $X$ which contains $f$. A function $f$ in  $X$  is said to be cyclic (or weakly invertible)  in $X$ if $[f]_{X}= X$. For some information on cyclic functions 
see \cite{BH} and the references therein.
In the case when $X=A^2(\DD)$ is the Bergman space, defined as
$$
A^2(\DD)=\bigl\{f\ \text{analytic in}\ \DD : \int_0^1\int_0^{2\pi}|f(re^{i\theta})|^2 d\theta<\infty \bigr\},
$$
a singular inner function $S_\mu$, 
$$
S_\mu(z):=\exp-\frac{1}{2\pi}\int_0^{2\pi} \frac{e^{i\theta}+z}{e^{i\theta}-z} d\mu(\theta), \qquad z\in\DD,
$$ 
is cyclic in $A^2(\DD)$  if and only if its associated positive singular measure $\mu$  places no mass on any $\Lambda$-Carleson set for $\Lambda(t)=\log(1/t)$. $\Lambda$-Carleson sets constitute 
a class of thin subsets of $\TT$, they will be discussed shortly. The necessity of this Carleson set condition was proved by H.~S.~Shapiro 
in 1967 \cite[Theorem~2]{HS}, and the sufficiency was proved independently by Korenblum in 1977 \cite{BK3} and Roberts in 1979 
\cite[Theorem~2]{JR}.

In the following a majorant $\Lambda$ will always denote a positive non-increasing convex differentiable function on $(0\text{,}1]$ such that:
\begin{itemize}
  \item $\Lambda(0)  =+\infty$
  \item  $t\Lambda(t)$ is a continuous, non-decreasing and concave function on $[0,1]$, and $t\Lambda(t)\to 0$
  as $t\to 0$. 
  \item There exists  $ \alpha \in (0,1)$ such that $t^{\alpha}\Lambda(t) $ is non-decreasing.
  \item
  \begin{equation}
  \Lambda(t^2)\le C\Lambda(t).
  \label{xf1}
  \end{equation}
\end{itemize}

Typical examples of majorants $\Lambda$ are $\log^+\log^+(1/x)$, 
$(\log(1/x))^p$, $p>0$.

In this work, we shall be interested mainly in studying cyclic vectors in the case $X=\mathcal{A}_{\Lambda }^{-\infty}$, generalizing the theory of premeasures developed by Korenblum; here $\mathcal{A}_{\Lambda }^{-\infty}$ is the  Korenblum type space associated with the majorant $\Lambda$, defined by
$$
\mathcal{A}_{\Lambda }^{-\infty}=\cup_{c>0}\mathcal{A}_{\Lambda }^{-c}=\bigcup_{c>0}\Big\{ f\in\Hol(\DD):|f(z)|\leq
\exp (c\Lambda(1-|z|))\Big\} .
$$
With the norm 
$$
\|f\|_{\mathcal{A}_{\Lambda }^{-c}}=\sup_{z\in\DD}|f(z)|\exp (-c\Lambda(1-r))<\infty,
$$
$\mathcal{A}_{\Lambda }^{-c}$ becomes a Banach space and for every $c_2\ge c_1>0$, the inclusion $\mathcal{A}_{\Lambda }^{-c_1}\hookrightarrow \mathcal{A}_{\Lambda }^{-c_2}$  is continuous. The topology on 
$$
\mathcal{A}_{\Lambda }^{-\infty}=\cup_{c>0}\mathcal{A}_{\Lambda }^{-c},
$$
is the locally-convex inductive limit topology, i.e. each of the inclusions $\mathcal{A}_{\Lambda }^{-c}\hookrightarrow \mathcal{A}_{\Lambda }^{-\infty}$ is continuous and the topology is the largest locally-convex topology with this property. A sequence $\{f_n\}_n\in\mathcal{A}_\Lambda^{-\infty }$ converges to 
$f\in\mathcal{A}_\Lambda^{-\infty}$ if and only if there exists $N>0$ such that all $f_n$ and $f$ belong to $\mathcal{A}_\Lambda^{-N}$, and $\lim_{n\to+\infty}\|f_n-f\|_{\mathcal{A}_\Lambda^{-N}}=0$.

The notion of a premeasure (a distribution of the first class) and the definition of the $\Lambda$-boundedness property of premeasure was first introduced in \cite{BK1}, for the case of $\Lambda(t)=\log(1/t)$  in connection with an extension of the Nevanlinna theory (see also   \cite{BK2} and \cite[Chapter~7]{HKZ}). 
Later on, in \cite{BK4}, Korenblum introduced a space of $\Lambda$-smooth functions 
and proved that the so called premeasures of bounded $\Lambda$-variation are the bounded linear functionals 
on this space. Next, he established that any premeasure of bounded $\Lambda$-variation is the difference of 
two $\Lambda$-bounded premeasures \cite[p.542]{BK4}. Finally, he described the Poisson integrals 
of $\Lambda$-bounded premeasures.

Our paper is organized as follows: In Section 2, we first introduce the notion of a $\Lambda$-bounded premeasure, and we will prove, using some arguments of real-variable theory, a general approximation theorem for $\Lambda$-bounded premeasures which will be critical for describing the cyclic vectors in $\mathcal{A}_{\Lambda }^{-\infty}$. Furthermore, this theorem shows that in respect to some general measure-theoretical properties, premeasure with vanishing $\Lambda$-singular part (see definition \ref{def sing}), behave themselves in some ways like absolutely continuous measures in the classical theory.

In Section 3, we  show that  every $\Lambda$-bounded premeasure $\mu$ generates a harmonic function $h(z)$ in $\DD$ (the Poisson integral of $\mu$) such that
\begin{equation}
h(z)= O(\Lambda (1-|z|)), \qquad |z|\to 1,\,z\in \DD,
\label{eqc}
\end{equation}
by the formula 
$$
h(z)=\int_{\TT} \frac{1-|z|^2}{|e^{i\theta}-z|^2}d\mu.
$$
Conversely, every real harmonic function $h(z)$  in $\DD$,  satisfying $h(0)=0$ and \eqref{eqc} is the Poisson integral of a 
$\Lambda$-bounded premeasure. (This result is formulated in \cite[p.543]{BK4} without proof, in a more general situation).

Finally, in Section 4 we characterize cyclic vectors in the spaces 
$\mathcal{A}_{\Lambda}^{-\infty}$ in terms of vanishing the 
$\Lambda$-singular part of the corresponding premeasure.
We prove two results for two different growth ranges of the majorant $\Lambda$. At the end we give two examples that show how the cyclicity 
property of a fixed function changes in a scale of 
$\mathcal{A}_{\Lambda_\alpha}$ spaces, $\Lambda_\alpha(x)=(\log (1/x))^\alpha$, $0<\alpha<1$.

Throughout the paper we use the following notation: given two functions $f$ and $g$ defined 
on $\Delta$ we write $f\asymp g$ if for some $0<c_1\le c_2<\infty$ we have  
$c_1f\le g\le c_2 f$ on $\Delta$.

\vspace{1em}

{\bf Acknowledgments:}  The author is  grateful to Alexander Borichev, Omar El-Fallah, and Karim Kellay for their useful comments and suggestions on this paper. Also the author would like to thank referee for the valuable comments.

\section{$\Lambda$-bounded premeasures}

In this section we extend the results of two papers by Korenblum \cite{BK1,BK2} on $\Lambda$-bounded premeasures (see also \cite[Chapter~7]{HKZ}) 
from the case $\Lambda(t)=\log(1/t)$ to the general case.

Let $\mathcal{B}(\TT)$ be the set of all (open, half-open and closed) arcs of $\TT$ including all the single points and the empty set. The elements of $\mathcal{B}(\TT)$ will be called intervals.

\begin{defi}\label{kappamesure}
A real function defined on $\mathcal{B}(\TT)$ is called a premeasure if the following conditions hold:
\begin{enumerate}
	\item $\mu(\TT)=0$
	\item $\mathcal{\mu}(I_{1}\cup I_{2})=\mathcal{\mu}(I_{1})+\mathcal{\mu}(I_{2}) $ for every $\ I_{1}$,
          $I_{2}\in\mathcal{B}(\TT)  $ such that $\ I_{1}\cap I_{2}=\emptyset$ and $I_{1}\cup I_{2}\in\mathcal{B}(\TT)$
	\item $\underset{n\to +\infty}{\lim}\mathcal{\mu}(I_{n})  =0$ for every sequence of embedded intervals, $I_{n+1}\subset I_{n}$, $n\ge 1$, such that\ 
$\bigcap_n I_{n}=\emptyset$.
\end{enumerate}
\end{defi}

Given a premeasure $\mu$, we introduce a real valued function $\overset{\curlywedge}{\mu}$ on $(0,2\pi]$ defined as follows: 
$$
\overset{\curlywedge}{\mu}(\theta)  =\mu(I_{\theta}),
$$
where 
$$
I_{\theta}=\big\{\xi\in\TT\text{ : }0\leq\arg\xi <\theta\big\}.
$$
The function $\overset{\curlywedge}{\mu}$ satisfies the following properties:
\begin{enumerate}
	\item[(a)] $\overset{\curlywedge}{\mu}(\theta^{-})$ exists for every 
               $\theta\in(0\text{,}2\pi]  $ and $\overset{\curlywedge}{\mu}( \theta^{+})$ 
               exists for every $\theta\in[0,2\pi)$
\item[(b)] $\overset{\curlywedge}{\mu}(\theta)=\lim_{t\to \theta^{-}}\overset{\curlywedge}{\mu}(t)$ for all 
           $\ \theta\in( 0\text{,}2\pi]$
\item[(c)] $\overset{\curlywedge}{\mu}(2\pi)=\lim_{\theta\to 0^{+}}\overset{\curlywedge}{\mu}(\theta)=0$.
\end{enumerate}
Furthermore, the function $\overset{\curlywedge}{\mu}(\theta)$ has at most countably many points of discontinuity.

\begin{defi}
A real premeasure $\mu$ is said to be $\Lambda$-bounded, if there is a positive number $C_{\mu}$ such that
\begin{equation}\label{bounded}
\mathcal{\mu}(I)\leq C_{\mu}| I| \Lambda(|I|)
\end{equation}
for any interval $I$.
\end{defi}

The minimal number $C_{\mu}$ is called the norm of $\mu$ and is denoted by $\| \mu\|^{+}_{\Lambda}$; 
the set of all real premeasures $\mu$ such that $\|\mu\|^{+}_{\Lambda}<+\infty$ is denoted by $\LB$.

\begin{defi}\label{def2}
A sequence of premeasures $\{\mu_{n}\}_{n}$ is said to be $\Lambda$-weakly convergent to a premeasure $\mu$ if :
\begin{enumerate}
	\item $\sup_n\|\mu_n\|^{+}_{\Lambda}<+\infty$, and 
	\item for every point $\theta$ of continuity of $\overset{\curlywedge}{\mu }$ we have $ \lim_{n\to\infty} \overset{\curlywedge}{\mu}_{n}(\theta)=\overset{\curlywedge}{\mu }(\theta).$
\end{enumerate}
\end{defi}
In this situation, the limit premeasure $\mu$ is $\Lambda$-bounded.
\vspace{1em}

Given a closed non-empty subset $F$ of the unit circle $\TT$, we define its $\Lambda$-entropy as follows:
$$
Entr_{\Lambda}(F)=\sum_{n} |I_{n}|\Lambda(|I_{n}|), 
$$
where $\{I_{n}\}_{n}$ are the component arcs of $\TT\setminus F$, and $|I|$ denotes the normalized Lebesgue measure of $I$ on $\TT$.
We set $Entr_{\Lambda}(\emptyset)=0$.

We say that a closed set $F$ is a $\Lambda$-Carleson set if $F$ is non-empty, has Lebesgue measure zero (i.e $|F|=0$), and 
$Entr_{\Lambda}(F)<+\infty$.

Denote by $\mathcal{C}_{\Lambda}$ the set of all $\Lambda$-Carleson sets and by $\mathcal{B}_{\Lambda}$ the set of all Borel sets $B\subset \TT$ such that $\overline{B} \in \mathcal{C}_{\Lambda}$.

\begin{defi} \label{def sing}
A function $\sigma : \mathcal{B}_{\Lambda} \rightarrow \R$ is called a $\Lambda$-singular measure if
\begin{enumerate}
	\item $\sigma$ is a finite Borel measure on every set in $\mathcal{C}_{\Lambda}$ (i.e. $\sigma{\bigm|}F$ is a Borel measure on $\TT$).
    \item There is a constant $C>0$ such that $$ |\sigma (F) |\leq C Entr_{\Lambda}(F)$$ for all $F \in \mathcal{C}_{\Lambda}$.
\end{enumerate}	
\end{defi}

Given a premeasure $\mu$ in $\LB$, its $\Lambda$-singular part is defined by :
\begin{equation}\label{eq1}
{\mu}_{s}(F)=-\sum_{n} \mu(I_{n}),
\end{equation}
where $F\in\mathcal{C}_{\Lambda}$ and $\{I_{n}\}_{n}$ is the collection of complementary intervals to $F$ in $\TT$. Using 
the argument in \cite[Theorem 6]{BK1} one can see that ${\mu}_{s}$ extends to a $\Lambda$-singular measure on $\mathcal{B}_{\Lambda}$.

\begin{proposition}\label{prop1}
If $\mu$ is a $\Lambda$-bounded premeasure, $F\in\mathcal{C}_\Lambda$, 
then $\mu_s {\bigm|}F$ is finite and non-positive.
\end{proposition}

\begin{proof}
Let $F\in\mathcal{C}_{\Lambda}$. We are to prove that ${\mu}_{s}(F) \leq 0$.


Let $\{I_{n}\}_{n}$ be the (possibly finite) sequence of the intervals complementary to $F$ in $\TT$. For $N\ge 1$, we 
consider the disjoint intervals $\{J^N_{n}\}_{1\le n\le N}$ such that $\TT\setminus \bigcup_{n=1}^{N} I_n=\bigcup_{n}^{N}J^N_n$. 
Then
$$
-\sum_{n= 1}^{N}\mu(I_{n})=\underset
{n=1}{\overset{N}{\sum}}\mu(J^N_{n})\leq\| \mu\|
^{+}_{\Lambda}\sum_{n= 1}^{N}|J^N_{n}|
\Lambda(|J^N_{n}|).
$$
Furthermore, each interval $J^N_n$ is covered by intervals $I_m\subset J^N_n$ up to a set of measure zero, and 
$\max_{1\le n\le N}|J^N_n|\to 0$ as $N\to\infty$ (If the sequence $\{I_{n}\}_{n}$ is finite, then all $J^N_n$ are single points 
for the corresponding $N$).
Therefore,
$$
-\sum_{n= 1}^{N}\mu(I_{n})  \leq\|
\mu\| ^{+}_{\Lambda}\sum_{n= 1}^{N}\underset{I_m\subset J^N_n}{\sum}|I_m|\Lambda (|I_m|)  \leq\| \mu\| ^{+}_{\Lambda}\underset{n>N}{\sum}|
I_{n}|\Lambda(|I_{n}|).
$$
Since $F$ is a $\Lambda$-Carleson set,
$$
-\lim_{N\to\infty}\sum_{n=1}^N\mu(I_n)\le 0.
$$
Thus, $\mu_s {\bigm|}F\le 0$.
\end{proof}

Given a closed subset $F$ of $\TT$, we denote by $F^\delta$ its $\delta$-neighborhood:
$$
F^\delta=\{\zeta\in \TT\text{ : }d(\zeta,F)\leq \delta\}.
$$

\begin{proposition}\label{prop}
Let $\mu$ be a $\Lambda$-bounded premeasure and let ${\mu}_{s}$ be its $\Lambda$-singular part. Then for every 
$F\in\mathcal{C}_{\Lambda}$ we have
\begin{equation}
\label{x22}
\mu_s(F)=\lim_{\delta\to 0}\mu(F^\delta).
\end{equation}
\end{proposition}

\begin{proof}
Let $F \in \mathcal{C}_{\Lambda}$, and let $\{I_{n}\}_{n},$ $|I_{1}|\geq|I_{2}|\geq\ldots,$ be the intervals 
of the complement to $F$ in $\TT$. We set 
$$
I_{n}^{(\delta)}=\big\{e^{i\theta} : \dist(e^{i\theta},\TT \setminus I_n)>\delta \big\}.
$$
Then for $|I_{n}|\geq 2\delta,$ we have 
$$
I_{n}=I_{n}^{1}\sqcup I_{n}^{(\delta)}\sqcup I_{n}^{2}
$$
with $|I_{n}^{1}|=|I_{n}^{2}|=\delta$. 
We see that
$$
\mu(F^{\delta})=-\underset{|I_{n}|> 2\delta}{\sum} \mu(I_{n}^{(\delta)}).
$$
Using relation \eqref{eq1} we obtain that 
\begin{eqnarray*}
 -{\mu}_{s}(F)&=&\sum_{n} \mu(I_{n})\\
 &=&\underset{|I_{n}|\le 2\delta}{\sum} \mu(I_{n})+\underset{|I_{n}|> 2\delta}{\sum} \Big[\mu(I_{n}^{1})+\mu(I_{n}^{(\delta)})+\mu(I_{n}^{2})\Big]\\
 &=&\underset{|I_{n}|\le 2\delta}{\sum} \mu(I_{n})-\mu(F^{\delta})+
 \underset{|I_{n}|> 2\delta}{\sum} \big[\mu(I_{n}^{1})+\mu(I_{n}^{2})\big].
\end{eqnarray*}
Therefore,
$$
\mu(F^{\delta})-{\mu}_{s}(F)=\underset{|I_{n}|\le 2\delta}{\sum} \mu(I_{n})+\underset{|I_{n}|> 2\delta}{\sum} \big[\mu(I_{n}^{1})+\mu(I_{n}^{2})\big]
$$
The first sum tends to zero as $\delta\to 0$, and it remains to prove that
\begin{equation}
\label{x21}
\lim_{\delta\to 0}\underset{|I_{n}|>2\delta}{\sum} \mu(I_{n}^{1})=0.
\end{equation}

We have
$$
\underset{|I_{n}|>2\delta}{\sum} \mu(I_{n}^{1})\le C\underset{|I_{n}|>\delta}{\sum} \delta\Lambda(\delta)=
C\underset{|I_{n}|>\delta}{\sum} \frac{\delta\Lambda(\delta)}{|I_{n}|\Lambda(|I_{n}|)}\cdot|I_{n}|\Lambda(|I_{n}|).
$$
Since the function $t\mapsto t\Lambda(t)$ does not decrease, we have
$$
\frac{\delta\Lambda(\delta)}{|I_{n}|\Lambda(|I_{n}|)}\le 1, \qquad |I_{n}|>\delta.
$$
Furthermore,
$$
\lim_{\delta\to 0}\frac{\delta\Lambda(\delta)}{|I_{n}|\Lambda(|I_{n}|)}=0, \qquad n\ge 1.
$$
Since
$$
\sum_{n\ge 1} |I_{n}|\Lambda(|I_{n}|)<\infty,
$$
we conclude that \eqref{x21}, and, hence, \eqref{x22} hold.

\end{proof}

\begin{defi}\label{def1}
A premeasure $\mu$ in $\LB$ is said to be $\Lambda$-absolutely continuous if there exists a sequence 
of $\Lambda$-bounded premeasures $(\mu_{n})_{n}$ such that:
\begin{enumerate}\label{ab-continuous}
	\item  
	$\sup_n \|
	\mu _{n}\|^{+}_{\Lambda}< +\infty $.
	\item $\sup_{I \in\mathcal{B}(\TT)} |(\mu+\mu_{n})(I)
|\to 0$ as $n\to +\infty$.
\end{enumerate}
\end{defi}

\begin{theo} \label{theo app }
Let $\mu$ be a premeasure in $\LB$. Then $\mu$ is $\Lambda$-absolutely continuous if and only if its 
$\Lambda$-singular part $\mu_{s}$ is zero.
\end{theo}

The only if part holds in a more general situation considered by Korenblum, \cite[Corollary, p.544]{BK4}.
On the other hand, the if part does not hold for differences of $\Lambda$-bounded premeasures 
(premeasures of $\Lambda$-bounded variation), see \cite[Remark, p.544]{BK4}.

To prove this theorem we need several lemmas. The first one is a linear programming lemma from \cite[Chapter 7]{HKZ}.

\begin{lem} \label{lem1} Consider the following system of $N(N+1)/2$ linear inequalities
in $N$ variables $x_1,\ldots,x_N$
\begin{equation*}
\sum_{j=k}^l x_j\le b_{k,l},\qquad 1\le k\le l\le N,
\end{equation*}
subject to the constraint: $x_1+x_2+\cdots+x_N=0$.
This system has a solution if and only if 
$$
\sum_n b_{k_{n},l_{n}}\ge 0
$$ 
for every simple covering  $\mathcal{P}=\{ [ k_{n},l_{n}]\}_n$ of $[1,N]$.
\end{lem}

The following lemma gives a necessary and sufficient conditions for a premeasure in $\LB$ to be $\Lambda$-absolutely continuous.

\begin{lem} \label{lem2}
Let $\mu$ be a $\Lambda$-bounded premeasure.
Then $\mu$ is $\Lambda$-absolutely continuous if and only if there is a positive constant $C>0$ such that for every $\varepsilon>0$  there exists a positive $M$ such that the system 
\begin{equation} \label{1.1}
\left\{
\begin{array}{cll}
x_{k,l}&\leq& M|I_{k,l}|\Lambda (|I_{k,l}|) \\
\mu (I_{k,l})+x_{k,l}&\leq& \min \{C|I_{k,l}|\Lambda (|I_{k,l}|), \varepsilon \} \\
x_{k,l}&=&\underset{s=k}{\overset{l-1}{\sum }}x_{s,s+1} \\
x_{0,N}&=&0
\end{array}
\right.
\end{equation}
in variables $x_{k,l},$ $0\leq k<l\leq N$, has a solution for every positive integer $N$. Here $I_{k,l}$ are the half-open arcs of 
$\TT$ defined by 
$$
I_{k,l}=\Bigl\{e^{i\theta }:  2\pi\frac{ k}{N}\leq \theta < 2\pi \frac{l}{N}\Bigr\}.
$$
\end{lem}

\begin{proof} Suppose that $\mu$ is $\Lambda$-absolutely continuous and denote by $\{\mu_{n}\}$ a sequence of $\Lambda$-bounded premeasures satisfying the conditions of Definition~\ref{ab-continuous}. Set
$$
C=\sup_n\|\mu+\mu_{n}\|^{+}_{\Lambda}, \qquad M=\sup_n\|\mu _{n}\|^{+}_{\Lambda},
$$
and let $\varepsilon>0$. For large $n$, the numbers $x_{k,l}=\mu_{n}(I_{k,l})$, $0\leq k<l\leq N$, satisfy relations \eqref{1.1} for all $N$.

Conversely, suppose that for some $C>0$ and for every $\varepsilon>0$ there exists $M=M(\varepsilon)>0$ 
such that for every $N$ there are $\{x_{k,l}\}_{k,l}$ (depending on $N$) satisfying relations \eqref{1.1}.
We consider the measures $d\mu_N$ defined on $I_{s,s+1}$, $0\le s<N$, by 
$$
d\mu_N(\xi)=\frac{x_{s,s+1}}{|I_{s,s+1}|}|d\xi|,
$$
where $|d\xi|$ is normalized Lebesgue measure on the unit circle $\TT$. 
To show that $\mu_{N}\in\LB$, it suffices to verify that the quantity
$\sup_{I}\frac{\mu(I)} {|I|\Lambda(|I|)}$ is finite for every interval $I\in\mathcal{B}(\TT)$.
Fix $I \in\mathcal{B}(\TT)$ such that $1\notin I$.

If $I\subset I_{k,k+1}$, then
$$
\mu_N(I)=\frac{x_{k,k+1}}{|I_{k,k+1}|}|I|
\le\frac{x_{k,k+1}}{|I_{k,k+1}|\Lambda( | I_{k,k+1}|) }| I|\Lambda( |I|)
\le M|I|\Lambda( |I|) \text{.}
$$

If $I=I_{k,l}$, then
$$
\mu _{N}(I_{k,l})=\underset{s=k}{\overset{l-1}{\sum }}\mu _{N}(I_{s,s+1})=
\underset{s=k}{\overset{l-1}{\sum }}x_{s,s+1}=x_{k,l}\leq M|I_{k,l}|\Lambda (
|I_{k,l}|).
$$

Otherwise, denote by $I_{k,l}$ the largest interval such that $I_{k,l}\subset I$. We have
\begin{eqnarray*}
\mu _{N}(I) &=&\mu _{N}(I_{k,l})+\mu _{N}(I\setminus I_{k,l}) \\
&\leq &M|I_{k,l}|\Lambda (|I_{k,l}|) +\max(x_{k-1,k},0)+\max(x_{l,l+1},0)
\\
&\leq &3M|I_{k,l}|\Lambda (|I_{k,l}|) \leq 3M|I|\Lambda(
|I|).
\end{eqnarray*}

Thus, $\mu_N$ is a $\Lambda$-bounded premeasure. Next, using a Helly-type selection theorem for premeasures 
due to Cyphert and Kelingos \cite[Theorem 2]{CK}, 
we can find a $\Lambda$-bounded premeasure $\nu$ and a subsequence $\mu_{N_{k}}\in\LB$ such that 
$\{\mu_{N_{k}}\}_k$ converge $\Lambda$-weakly to $\nu$. Furthermore, $\nu$ satisfies the following conditions:

$\nu (J)\leq 3M|J|\Lambda(|J |)$ and
$\mu (J)+\nu (J)\leq \min \{C|J|\Lambda(|J|),\varepsilon \}$ 
for every interval $J\subset \TT\setminus\{1\}$.

Now, if $I$ is an interval containing the point $1$, we can represent it as $I=I_{1}\sqcup \{1\} \sqcup I_{2}$, for some 
(possibly empty) intervals $I_{1}$ and $I_{2}$. Then 
\begin{eqnarray*}
\mu(I)+\nu(I) & =&(\mu+\nu)(I_{1})+(\mu+\nu)(I_{2})+(\mu+\nu)(\{1\})\\
&\leq & (\mu+\nu)(I_{1})+(\mu+\nu)(I_{2}).
\end{eqnarray*}
Therefore, for every $I\in\mathcal{B}(\TT)$ we have $\mu (J)+\nu (J)\le 2\varepsilon$.
Since $(\mu+\nu)(\TT\setminus I)=-\mu(I)-\nu(I)$, we have 
$$
|\mu (J)+\nu (J)|\le 2\varepsilon.
$$
Thus $\mu$ is $\Lambda$-absolutely continuous.
\end{proof}

\begin{lem}\label{deduce}
Let $\mu\in \LB$ be not $\Lambda$-absolutely continuous. Then for every $C>0$ there is $\varepsilon>0$ such that for all $M>0$, 
there exists a simple covering of $\TT$ by a finite number of half-open intervals $\{I_{n}\}_{n}$, satisfying the relation 
$$
\sum_{n}\min \left\{ \mu (I_{n})+M|I_{n}|\Lambda(|I_{n}|) ,C|I_{n}|\Lambda (|I_{n}|),\varepsilon \right\} <0.
$$
\end{lem}

\begin{proof} By Lemma~\ref{lem2}, for every $C>0$ there exists a number $\varepsilon>0$ such that for all $M>0$, 
the system \eqref{1.1} has no solutions for some $N\in\N$. In other words, there are no $\{x_{k,l}\}_{k,l}$ such that:
\begin{equation}
\label{x46}
\overset{l-1}{\underset{s=k}{\sum }}\mu( I_{s,s+1})
+x_{s,s+1}\leq \min \big\{ \mu (I_{k,l})+M|I_{k,l}|\Lambda (|
I_{k,l}|) ,C|I_{k,l}|\Lambda ( |I_{k,l}|)
,\varepsilon \big\} 
\end{equation}
with $x_{k,l}=\underset{s=k}{\overset{l-1}{\sum }}
x_{s,s+1}$ and $x_{0,N}=0$.

We set $X_{j}=\mu (I_{j,j+1})+x_{j,j+1}$, and
$$
b_{k,l}=\min \big\{ \mu (I_{k,l+1})+M|I_{k,l+1}|\Lambda (|I_{k,l+1}|) ,C|I_{k,l+1}|\Lambda ( |I_{k,l+1}|) ,\varepsilon\big\}.
$$
Then relations \eqref{x46} are rewritten as
$$
\underset{j=k}{\overset{l}{\sum}}X_{j}\leq b_{k,l}, \qquad 0\leq k<l\leq N-1.
$$
Therefore, we are in the conditions of Lemma~\ref{lem1} with variables $X_{j}$. We conclude that there is a simple covering of 
the circle $\TT$ by a finite number of half-open intervals $\{I_{n}\}$ such that
$$\sum_{n}\min \big\{ \mu (I_{n})+M|I_{n}|\Lambda(|
I_{n}|) ,C|I_{n}|\Lambda (|I_{n}|)
,\varepsilon \big\} <0.
$$
\end{proof}

In the following lemma we give a normal families type result for the  $\Lambda$-Carleson sets.

\begin{lem}\label{normal.f}
Let $\{F_n\}_{n}$ be a sequence of sets on the unit circle, and let each $F_n$ be a finite union of closed intervals. We assume that
\begin{enumerate}
\item[(i)] $|F_{n}|\to 0$,\qquad $n\to\infty$,
\item[(ii)] $Entr_{\Lambda} (F_{n})=O(1)$,\qquad $n\to\infty$.
\end{enumerate}
Then there exists a subsequence $\{F_{n_{k}}\}_{k}$ and a $\Lambda$-Carleson set $F$ such that:

for every $\delta >0$ there is a natural number $N$ with
\begin{enumerate}
\item[(a)] $F_{n_{k}}\subset F^{\delta }$

\item[(b)] $F\subset F_{n_{k}}^{\delta }.$
\end{enumerate}
for all $k\ge N$.
\end{lem}

\begin{proof}
Let $\{I_{k,n}\}_k$ be the complementary arcs to $F_{n}$ such that $|I_{1,n}|\ge|I_{2,n}|\ge\ldots$.
We show first that the sequence $\{|I_{1,n}|\}_{n}$ is bounded away from zero.  
Since the function $\Lambda$ is non-increasing, we have
\begin{equation*}
Entr_{\Lambda}(F_{n})=\sum_{k}|
I_{k,n}|\Lambda (|I_{k,n}|)\geq |
\TT\setminus F_{n}|\Lambda (|
I_{1,n}|),
\end{equation*}
and therefore,
\begin{equation*}
\frac{Entr_{\Lambda} (F_{n})}{|
\TT\setminus F_{n}|}\geq \Lambda (|
I_{1,n}|).
\end{equation*}
Now the conditions (i) and (ii) of lemma and the fact that $\Lambda(0^{+})=+\infty$ imply that the sequence $\{|I_{1,n}|\}_{n}$ is bounded away from zero.

Given a subsequence $\{F_{k}^{(m)}\}_k$ of $F_{n}$, we denote by $(I_{j,k}^{{(m) }})_{j}$ the complementary arcs to $F_k^{(m)}$. 
Let us choose a subsequence $\{F_k^{(1)}\}_k$ such that 
$$
I_{1,k}^{{(1) }}=(a_k^{(1)},b_k^{(1)}) \to (a^1,b^1)=J_1
$$ 
as $k\to +\infty$, where $J_{1}$ is a non-empty open arc.

If $|J_{1}|=1$, then $F=\TT\setminus J_{1}$ is a $\Lambda$-Carleson set, and we are done: we can take $\{F_{n_{k}}\}_{k}=\{F_{k}^{(l)}\}_k$.

Otherwise, if $|J_1|<1$, then, using the above method we show that
$$
\Lambda(|I_{2,k}^{(1)}|)\le\frac{Entr_{\Lambda}(F_{k}^{(1)})}{|\TT\setminus F_{k}^{{(1)}}|-|I_{1,k}^{(1)}|}.
$$

Since $\lim_{k\to +\infty}|\TT\setminus F_{k}^{{(1)}}|-|I_{1,k}^{{(1)}}|=1-|J_{1}|>0$, 
the sequence $\Lambda (|I_{2,k}^{{(1)}}|)$ is bounded, and hence, the sequence $|I_{2,k}^{{(1)}}|$ is bounded away from zero.
Next we choose a subsequence $\{F_{k}^{{(2)}}\}_{k}$ of $\{F_{k}^{{(1)}}\}_{k}$ such that the arcs $I_{2,k}^{(2)}=(a_{k}^2,b_{k}^2)$ tend to $(a^{(2)},b^{(2)})=J_2$, where $J_2$ is a non-empty open arc.
Repeating this process we can have two possibilities.
First, suppose that after a finite number of steps we have $|J_1|+\ldots+|J_m|=1$, and then 
we can take $\{F_{n_{k}}\}_{k}=\{F_{k}^{(m)}\}_k$, 
$$
I_{j,k}^{{(m)}}\to J_j,\qquad 1\le j\le m,
$$ 
as $k\to +\infty$, and $F=\TT\setminus \underset{j=1}{\overset{m}{\cup }}J_j$ is $\Lambda $-Carleson.

Now, if the number of steps is infinite, then using the estimate 
$$
\Lambda(|J_{l}|)\leq \frac{\sup_n \big\{Entr_{\Lambda}(F_{n})\big\}}{1-\sum_{k=1}^{l-1}|J_{k}|},
$$ 
and the fact $|J_m|\to 0$ as $m\to \infty$, we conclude that
$$
\sum_{j=1}^{\infty}| J_j|=1.
$$ 
We can set $\{F_{n_{k}}\}_{k}=\{F_{m}^{(m)}\}_m$, $F=\TT\setminus\bigcup_{j\ge 1} J_j$.

In all three situations the properties (a) and (b) follow automatically.
\end{proof}

\subsection*{Proof of Theorem \ref{theo app }}

First we suppose that $\mu$ is $\Lambda$-absolutely continuous, and prove that $\mu_{s}=0$. 
Choose a sequence $\mu_n$ of $\Lambda$-bounded premeasures satisfying the properties (1) and (2) of Definition~\ref{def1}. 
Let $F$ be a $\Lambda$-Carleson set and let $(I_{n})_n$ be the sequence of the complementary arcs to $F$. Denote by 
$(\mu+\mu_{n})_{s}$ the $\Lambda$-singular part of $\mu +\mu _{n}$.
Then
\begin{eqnarray*}
-(\mu +\mu _{n})_{s}(F)&=& \sum_{k}(\mu +\mu _{n})(I_{k})\\
&=&\sum_{k\leq N}(\mu +\mu _{n})(I_{k})+\sum_{k>N}(\mu +\mu _{n})(I_{k}) \\
&\leq &\sum_{k\leq N}(\mu +\mu _{n})(I_{k})+C\sum_{k>N}
| I_{k}| \Lambda (|I_{k}|)
\end{eqnarray*}
Using the property (2) of Definition~\ref{def1} we obtain that 
$$
-\liminf_{n\to\infty}(\mu +\mu _{n})_{s}(F)\leq C\sum_{k>N} |I_{k}|\Lambda (|I_{k}|).
$$
Since $F \in \mathcal{C}_{\Lambda}$, we have $\sum_{k>N} |I_{k}|\Lambda (|I_{k}|)\to 0$ as $N\to+\infty$,
and hence $\liminf_{n\to\infty}(\mu +\mu _{n})_{s}(F) \ge 0$. Since $(\mu +\mu _{n}) \in \LB$, by Proposition~\ref{prop1} its $\Lambda$-singular part is non-positive. Thus $\lim_{n\to\infty}(\mu +\mu _{n})_{s}(F)=0$ for all $F \in \mathcal{C}_{\Lambda}$,
which proves that $\mu _{s}=0$.
\medskip

Now, let us suppose that $\mu $ is not $\Lambda $-absolutely continuous. We apply Lemma~\ref{deduce} with $C=4\|\mu \|^{+}_{\Lambda}$  and find $\varepsilon >0$ such that for all $M>0$, there is a simple covering of circle $\TT$ by a half-open intervals 
$\{I_1,I_2,\ldots I_N\}$ such that
\begin{equation}\label{eq3}
\sum_n\min \left\{ \mu (I_{n})+M|I_{n}|\Lambda(|I_{n}|),4\|\mu \|^{+}_{\Lambda}|I_{n}|\Lambda (|I_{n}|),\varepsilon \right\} <0.
\end{equation}
Let us fix a number $\rho>0$ satisfying the inequality $\rho \Lambda (\rho )\leq{\varepsilon}/{4\|\mu \|^{+}_{\Lambda}}$. 
We divide the intervals $\{I_{1}, I_{2},\ldots I_{N}\}$ into two groups. The first group  
$\{I_{n}^{(1)}\}_n$ consists of intervals $I_n$ such that
\begin{equation}\label{eq7}
\min \{ \mu (I_{n})+M|I_{n}|\Lambda(
|I_{n}|) ,4\|\mu \|^{+}_{\Lambda}|I_{n}|\Lambda(|I_{n}|) ,\varepsilon \} =\mu (I_{n})+M|I_{n}|\Lambda(|I_{n}|),
\end{equation}
and the second one is $\{ I_n^{(2)}\}_n=\{I_n\}_n\setminus\{I_n^{(1)}\}_n$. 

Using these definitions and the fact that $\Lambda$ is non-increasing, we rewrite inequality \eqref{eq3} as
\begin{multline}
\sum_n \mu ( I_{n}^{(1)})+M \sum_{n} |I_{n}^{(1)}|\Lambda(|I_{n}^{(1)}|) \\<-4\|\mu \|^{+}_{\Lambda} \sum_{n \text{ : }| I_{n}^{(2)}|
<\rho} | I_{n}^{(2)}|\Lambda(| I_{n}^{(2)}|) -\varepsilon \;  \Card\{ n\text{ : }|  I_{n}^{(2)} |\geq\rho\}.\label{sous-aditive}
\end{multline}

Next we establish three properties of these families of intervals. From now on we assume that $M>4\|\mu \|^{+}_{\Lambda}$.

(1) We have $\{I_{n}^{(2)}:|I_{n}^{(2)}|\ge\rho\}\neq \emptyset$. Otherwise, by \eqref{sous-aditive}, we would have 
\begin{eqnarray*}
0&=&\mu(\TT) =\sum_{n} \mu ( I_{n}^{(1)})+\sum_{n} \mu (I_{n}^{(2)})  \\
&\leq&-M\sum_{n}|  I_{n}^{(1)}| \Lambda
(|  I_{n}^{(1)}| )-4\|\mu \|^{+}_{\Lambda}\sum_{n}|
 I_{n}^{(2)}| \Lambda (| I_{n}^{(2)}|
)+ \|\mu \| ^{+}_{\Lambda}\sum_{n}|
 I_{n}^{(2)}| \Lambda (| I_{n}^{(2)}| ) \\
&\leq& -M\sum_{n}|
 I_{n}^{(1)}| \Lambda (| I_{n}^{(1)}| )-3\| \mu \|^{+}_{\Lambda}\sum_{n}
| I_{n}^{(2)}| \Lambda (|
 I_{n}^{(2)}| )<0.
\end{eqnarray*}

(2) We have $\sum_n|I_{n}^{(2)}| \Lambda (|I_{n}^{(2)}| )\leq 2\Lambda (\rho )$.
To prove this relation, we notice first that for every simple covering $\{J_{n}\}_{n}$ of $\TT$, we have 
$$
0=\mu(\TT)=\sum_n\mu(J_n)=\sum_n\mu(J_n)^{+}-\sum_n\mu(J_n)^{-},
$$
and hence, 
$$
\sum_n|\mu(J_n)|=\sum_n\mu(J_n)^{+}+\sum_n\mu(J_n)^{-}=2\sum_n\mu(J_n)^{+}\le 2\|\mu\|^{+}_{\Lambda}\sum_n|J_{n}|\Lambda(|J_n|).
$$
Applying this to our simple covering, we get
$$
\sum_{n} |\mu ( I_{n}^{(1)})|+\sum_{n} |\mu (I_{n}^{(2)})|\leq 2\|\mu\|^{+}_{\Lambda}\underset{n}{\sum}\Big[|I_{n}^{(1)}| \Lambda (| I_{n}^{(1)}|)+|I_{n}^{(2)}| \Lambda (| I_{n}^{(2)}| )\Big],
$$
and hence,
$$
-\underset{n}{\sum}\mu ( I_{n}^{(1)}) \leq 2\|\mu\|^{+}_{\Lambda}\underset{n}{\sum}\Big[|I_{n}^{(1)}| \Lambda (| I_{n}^{(1)}|)+|I_{n}^{(2)}| \Lambda (| I_{n}^{(2)}| )\Big].
$$ 
Now, using \eqref{sous-aditive} we obtain that
$$
M \underset{n}{\sum}|I_{n}^{(1)}| \Lambda (|I_{n}^{(1)}| )+4\|\mu\|^{+}_{\Lambda}\underset{|
 I_{n}^{(2)}| <\rho }{\sum }| I_{n}^{(2)}|
\Lambda (| I_{n}^{(2)}|) \leq 2\|\mu\|^{+}_{\Lambda}\underset{n}{\sum}\Big[|I_{n}^{(1)}| \Lambda (| I_{n}^{(1)}|)+|I_{n}^{(2)}| \Lambda (| I_{n}^{(2)}| )\Big],
$$
and hence,
\begin{equation}\label{eq4}
\Big(M-2\|\mu\|^{+}_{\Lambda}\Big)\underset{n}{\sum}
|I_{n}^{(1)}| \Lambda (|I_{n}^{(1)}| )\leq 2\|\mu\|^{+}_{\Lambda}\bigg[\underset{|
 I_{n}^{(2)}| \geq \rho }{\sum }|I_{n}^{(2)}| \Lambda (| I_{n}^{(2)}|)-\underset{|
 I_{n}^{(2)}| <\rho }{\sum }| I_{n}^{(2)}|
\Lambda (| I_{n}^{(2)}|)\bigg].
\end{equation}
As a consequence, we have
$$
\underset{|I_{n}^{(2)}| <\rho }{\sum }| I_{n}^{(2)}|\Lambda (| I_{n}^{(2)}| )\leq \underset{|I_{n}^{(2)}| \geq \rho }{\sum }|I_{n}^{(2)}| \Lambda (| I_{n}^{(2)}|),
$$
and, finally,
$$
\underset{n}{\sum}|I_{n}^{(2)}| \Lambda (|I_{n}^{(2)}|) \leq 2\underset{|I_{n}^{(2)}| \geq \rho }{\sum }|I_{n}^{(2)}| \Lambda (| I_{n}^{(2)}|)\leq 2\underset{n}{\sum }|I_{n}^{(2)}|\Lambda(\rho)\leq 2\Lambda(\rho).
$$

(3) We have 
$$
\sum_{n}|I_{n}^{(1)}| \Lambda (|I_{n}^{(1)}| )\leq \frac{2\|\mu\|^{+}_{\Lambda}}{M-2\|\mu\|^{+}_{\Lambda}}\cdot \Lambda(\rho).
$$ 
This property follows immediately from \eqref{eq4}.

We set $F_{M}=\bigcup_{n}\overline{I_{n}^{(1)}}$. Inequality \eqref{sous-aditive} and the properties (1)--(3) show that
\begin{enumerate}
\item[(i)] $Entr_{\Lambda} (F_{M})=O(1)$, \quad $M\to\infty$,
\item[(ii)] $|F_{M}|\Lambda (|F_{M}|)\leq\frac{2\|\mu\|^{+}_{\Lambda}}{M-2\|\mu\|^{+}_{\Lambda}}\cdot \Lambda(\rho)$,
\item[(iii)] $\mu (F_{M})\leq -4\|\mu \|^{+}_{\Lambda}\Bigl[\underset{n}{\sum}|I_{n}^{(1)}|\Lambda (|I_{n}^{(1)}|)+\sum_{n \text{ : }|I_{n}^{(2)}|<\rho}|I_{n}^{(2)}|\Lambda (|I_{n}^{(2)}|)\Bigr] -\varepsilon$.
\end{enumerate}
By Lemma~\ref{normal.f} there exists a subsequence $M_{n}\to +\infty$ such that $F^*_n:=F_{M_n}$ (composed of a finite number of closed arcs) converge to a $\Lambda$-Carleson set $F$. More precisely, $F\subset {F^*_n}^{\delta}$ and $F^*_n\subset F^{\delta }$ for every 
fixed $\delta >0$ and for sufficiently large $n$. Furthermore, (iii) yields
\begin{equation}\label{eq5}
\mu (F^*_{n})\leq -4\|\mu \|^{+}_{\Lambda}\bigg[\underset{k}{\sum}|R_{k,n}|\Lambda(|R_{k,n}| )+\sum_{k: |L_{k,n}| <\rho }|L_{k,n}|\Lambda (|L_{k,n}| )\bigg]-\varepsilon,
\end{equation}
where $F^*_n=\bigsqcup_{k}R_{k,n}$ and $\TT\setminus F^*_n=\bigsqcup_{k}L_{k,n}$.

It remains to show that 
$$
\mu_s(F)<0.
$$
Otherwise, if $\mu_s(F)=0$, then by Proposition~\ref{prop} we have
$$
\lim_{\delta\to0} \mu (F^\delta)=0.
$$ 
Modifying a bit the set $F^*_n$, if necessary, we obtain $\lim_{\delta\to0} \mu (F^*_n\cap F^\delta)=0$.
Now we can choose a sequence $\delta_n>0$ rapidly converging to $0$ and a sequence $\{k_n\}$ rapidly converging to $\infty$ 
such that the sets $F_n$ defined by 
$$
F_n=F^*_{k_n}\setminus 
{F^{\delta_{n+1}}}\subset F^{\delta_n}\setminus 
{F^{\delta_{n+1}}},
$$ 
and consisting of a finite number of intervals $\{I_{k,n}\}_k$ satisfy the inequalities
\begin{equation}\label{eq6}
\mu(F_n)\le-4\|\mu\|^{+}_{\Lambda}\bigg[\underset{k}{\sum}|I_{k,n}|\Lambda (|I_{k,n}|) +\underset{k}{\sum}|J_{n,k}|\Lambda(|J_{n,k}|)\bigg]-\varepsilon/2,
\end{equation}
where
$\bigsqcup_k J_{n,k}=(F^{\delta _{n}}\setminus 
{F^{\delta_{n+1}}})\setminus F_n=:G_n$.

We denote by $\mathcal{I}_n$, $\mathcal{J}_n$, and $\mathcal{K}_n$ the systems of intervals that form 
$F_n$, $G_n$, and $F^{\delta_n}$, respectively. Furthermore, we denote by $\mathcal{I}_{0}$ be the system of intervals complementary  
to $F^{\delta_1}$, and we put $\mathcal{S}_n=(\underset{k=1}{\overset{n}{\cup }}\mathcal{I}
_{k})\cup (\underset{k=1}{\overset{n}{\cup }}\mathcal{J}_{n})\cup \mathcal{K}
_{n+1}$. Summing up the estimates on $\mu(F_n)$ in \eqref{eq6} we obtain
\begin{eqnarray}
\label{estimation}\sum_{I\in \mathcal{I}_{0}}| \mu (I)| +
\sum_{I\in \mathcal{S}_{n}}| \mu (I)|  &\geq &
\sum_{i=1}^{n}| \mu (F_{i})|  \nonumber\\
&\geq &4\|\mu\|_{\Lambda}^{+}\sum_{i=1}^{n}\bigg[\underset{k}{\sum}|I_{i,k}|\Lambda (|I_{i,k}|) +\underset{k}{\sum}| J_{i,k}|\Lambda (| J_{i,k}| )\bigg]+n\varepsilon/2 \nonumber \\
&=&4\|\mu\|_{\Lambda}^{+}\sum_{I\in \mathcal{S}_{n}}| I|\Lambda (| I| )-4\|\mu\|_{\Lambda}^{+}\sum_{I\in \mathcal{K}_{n+1}} | I|\Lambda (| I|)+n\varepsilon/2 \nonumber \\
&=&4\|\mu\|_{\Lambda}^{+}\Bigg[\sum_{I\in \mathcal{S}_{n}\cup \mathcal{I}_{0}} | I|\Lambda
(| I| )-\sum_{I\in \mathcal{K}_{n+1}}| I|\Lambda
(| I| )\Bigg] \nonumber \\
&-&4\|\mu\|_{\Lambda}^{+} \sum_{I\in \mathcal{I}_{0}}| I|\Lambda
(| I| )+n\varepsilon/2.
\end{eqnarray}

Notice that 
$$
\sum_{I\in\mathcal{K}_{n+1}}|I|\Lambda(|I|)\le 
\sum_{|J_k|<2\delta_{n+1}}|J_k|\Lambda(|J_k|)+2\delta_{n+1}\Lambda(\delta_{n+1})\cdot \Card\{k:|J_k|\ge 2\delta_{n+1}\},
$$
where $\{J_{k}\}_{k}$, $|J_{1}|\geq|J_{2}|\geq \ldots$ are the complementary arcs to the $\Lambda$-Carleson set $F$. 
Since $\lim_{t\to 0}t\Lambda(t)=0$, we obtain that 
$$
\lim_{n\to+\infty}\sum_{I\in \mathcal{K}_{n+1} }|I|\Lambda (|I|)=0.$$
Thus for sufficiently large $n$, \eqref{estimation} gives us the following relation
$$
\underset{I\in \mathcal{S}_{n}\cup \mathcal{I}_{0}}{\sum }| \mu (I)| \geq
4\|\mu\| ^{+}_{\Lambda}\underset{I\in \mathcal{S}_{n}\cup \mathcal{I
}_{0}}{\sum }|I|\Lambda (|I|)
$$
where $\mathcal{S}_{n}\cup \mathcal{I}_{0}$ is a simple covering of the unit circle. 
However, since $\mu\in \LB$, we have 
$$
\underset{I\in \mathcal{S}_{n}\cup \mathcal{I}_{0}}{\sum }| \mu (I)|=2\underset{I\in \mathcal{S}_{n}\cup \mathcal{I}_{0}}{\sum }
\max(\mu (I), 0)\le 2\|\mu\| ^{+}_{\Lambda}\underset{I\in \mathcal{S}_{n}\cup \mathcal{I
}_{0}}{\sum }|I|\Lambda (|I|).
$$
This contradiction completes the proof of the theorem.\qed
\vspace{1em}

\section{Harmonic functions of restricted growth}

Every bounded harmonic function can be represented via the Poisson integral of its boundary values. 
In the following theorem we show that a large class of real-valued harmonic functions in the unit disk $\DD$ 
can be represented as the Poisson integrals of $\Lambda$-bounded premeasures. Before formulating the main result of this section, let us introduce some notations.

\begin{defi}
Let $f$ be a function in $C^{1}(\TT)$ and let $\mu\in\LB$. We define the integral of the function $f$ with respect to $\mu$ by the formula
$$
\int_{\TT}f\,d\mu=\int^{2\pi}_{0}f(e^{it})\,d\overset{\curlywedge}{\mu}(t).
$$
In particular, we have
$$
\int_{0}^{2\pi }\frac{1-|z|^{2}}{|e^{i\theta }-z|^{2}}\,d\mu ( \theta ) =-\int_{0}^{2\pi}\Big(\frac{\partial}{\partial\theta}\frac{1-|z|^{2}}{|e^{i\theta }-z| ^{2}}\Big)\overset{\curlywedge}{\mu}(  \theta)\,d \theta.
$$
Given a $\Lambda$-bounded premeasure $\mu$ we denote by $P[\mu]$ its Poisson integral:
$$
P[\mu](z)=\int_{0}^{2\pi }\frac{1-|z|^{2}}{|e^{i\theta }-z|^{2}}\,d\mu( \theta).
$$
\end{defi}

\begin{proposition}\label{har-rep}
Let $\mu\in\LB$. The Poisson integral $P[\mu]$ satisfies the estimate
$$
P[\mu](z)\le 10\|\mu\|_{\Lambda}^{+}\Lambda(1-|z|),\qquad z\in\mathbb D.
$$
\end{proposition}

\begin{proof} It suffices to verify the estimate on the interval $(0,1)$. Let $0<r<1$. Then 
\begin{eqnarray*}
    P[\mu](r)&=&\int_{0}^{2\pi }\frac{1-r^{2}}{|e^{i\theta }-r|^{2}}\,d\mu(\theta)\\ \\
			 &=&-\int_{0}^{2\pi}\Big[\frac{\partial}{\partial\theta}\bigl(\frac{1-r^{2}}{|e^{i\theta}-r|^{2}}\bigr)\Big]\overset{\curlywedge}{\mu}(\theta)\,d \theta \\ \\
			 &=&\int_{0}^{2\pi}\frac{2r(1-r^{2})\sin\theta}{(1-2r\cos\theta+r^{2})^{2}}\mu(I_{\theta})\,d\theta\\ \\ 
			 &=&\int_{0}^{\pi}\frac{2r(1-r^{2})\sin\theta}{(1-2r\cos\theta+r^{2})^{2}}\mu(I_{\theta})\,d\theta
-\int_{\pi}^{0}-\frac{2r(1-r^{2})\sin\theta}{(1-2r\cos\theta+r^{2})^{2}}\mu(I_{2\pi-\theta})\,d\theta\\ \\
			 &=&\int_{0}^{\pi}\frac{2r(1-r^{2})\sin\theta}{(1-2r\cos\theta+r^{2})^{2}}\Big[\mu(I_{\theta})+\mu([-\theta,0))\Big]\,d\theta\\ \\
			 &=&\int_{0}^{\pi}\frac{2r(1-r^{2})\sin\theta}{(1-2r\cos\theta+r^{2})^{2}}\mu([-\theta,\theta))\,d\theta.
\end{eqnarray*}

Integrating by parts and using the fact that $\Lambda$ is decreasing and $t\Lambda(t)$ is increasing we get
\begin{multline*}
P[\mu](r)\le\|\mu\|_{\Lambda}^{+}\Lambda(1-r)\bigg[(1-r)\int_{0}^{\frac{1-r}{2}}\frac{2r(1-r^{2})\sin\theta}{(1-2r\cos\theta+r^{2})^{2}}d\theta
-\int_{\frac{1-r}{2}}^{\pi}2\theta\Big[\frac{\partial}{\partial\theta}\bigl(\frac{1-r^{2}}{|e^{i\theta}-r|^{2}}\bigr)\Big]\,d\theta\bigg]\\ \\
\le\|\mu\|_{\Lambda}^{+}\Lambda(1-r)\bigg[2(1-r)^{3}\int_{0}^{\frac{1-r}{2}}\frac{d\theta}{(1-r)^4}+\frac{(1-r)(1-r^{2})}{(1-r)^{2}}+2\int_{0}^{\pi}\frac{1-r^{2}}{|e^{i\theta}-r|^{2}}d\theta\bigg]\\ \\
\le 10\|\mu\|_{\Lambda}^{+}\Lambda(1-r).
\end{multline*}
\end{proof}

The following theorem is stated by Korenblum in \cite[Theorem~1, p.~543]{BK4} without proof, 
in a more general situation.

\begin{theo}\label{harmonic}
Let $h$ be a real-valued harmonic function on the unit disk such that $h(0)=0$ and 
\begin{equation*}
h(z)= O(\Lambda (1-|z|)), \qquad |z|\to 1,\,z\in \DD.
\end{equation*}
Then the following statements hold.
\begin{enumerate}
	\item For every open arc $I$ of the unit circle $\TT$ the following limit exists:
$$
\mu(I) =\underset{r\to 1^{-}}{\lim }\mu _{r}(I)
=\underset{r\to 1^{-}}{\lim }\int_{I}h(r\xi) \,|d\xi|<\infty.
$$
  \item $\mu$ is a $\Lambda$-bounded premeasure.
  \item The function $h$ is the Poisson integral of the premeasure $\mu$:
$$
h(z) =\int_{0}^{2\pi }\frac{1-|z|^{2}}{|e^{i\theta }-z|^{2}}\,d\mu (\theta), \qquad z\in \DD.
$$
\end{enumerate}
\end{theo}

\begin{proof} Let
$$
h( re^{i\theta }) =\sum_{n=-\infty}^{+\infty}a_{n}r^{|n|}e^{in\theta }.
$$   
Since $a_0=h(0)=0$, we have 
$$
\int_{0}^{2\pi }h^{+}(re^{i\theta })\,d\theta=\int_{0}^{2\pi }h^{-}( re^{i\theta})\,d\theta=
\frac{1}{2}\int_{0}^{2\pi }| h( re^{i\theta})|\,d\theta.
$$
Furthermore,
\begin{eqnarray}
| a_{n}| 
&=&\Big|\frac{r^{-| n| }}{2\pi }\int_{0}^{2\pi } h( re^{i\theta })e^{-in\theta}\,d\theta \Big|\notag\\  
&\le&\frac{r^{-| n| }}{2\pi }\int_{0}^{2\pi }| h( re^{i\theta })|\,d\theta =
\frac{r^{-| n| }}{\pi }\int_{0}^{2\pi}h^{+}( re^{i\theta })\,d\theta  \notag\\
&\le& Cr^{-| n| }\Lambda( 1-r) \notag \\
&\le& C_1\Lambda\bigl(\frac1{|n|}\bigr),\qquad \frac1{|n|}=1-r,\,n\in\mathbb{Z}\setminus\{-1,0,1\}.\label{xf2}
\end{eqnarray}

Let $I=\{e^{i\theta}:\alpha\le\theta\le\beta\}$ be an arc of $\TT$, $\tau=\beta-\alpha$.
For $\theta\in[\alpha,\beta]$ we define 
$$
t(\theta)=\min\{\theta-\alpha,\beta-\theta\}, \qquad 
\eta(\theta)=\frac{1}{\tau}(\beta-\theta)(\theta-\alpha).
$$
Then
\begin{equation*}
\frac{1}{2}\,t(\theta) \le \eta (\theta) \le t( \theta),\qquad |\eta ^{\prime}( \theta)|\le 1,
\qquad \eta''( \theta)= \frac{-2}{\tau},\qquad \theta\in[\alpha,\beta].
\end{equation*}
Given $p>2$ we introduce the function $q(\theta)=1-\eta(\theta)^p$ 
satisfying the following properties:
$$
|q'( \theta )|\le p\eta(\theta)^{p-1},\qquad |q''(\theta)|\le p^2\eta(\theta)^{p-2},\qquad \theta\in(\alpha,\beta).
$$

Integrating by parts we obtain for $|n|\ge1$ and $\tau <1$ that
\begin{multline*}
\Big|\int_{\alpha}^{\beta}(1-q(\theta)^{|n|}) e^{in\theta}\,d\theta\Big|
=\frac1{|n|}\Big|\int_\alpha^\beta|n|q(\theta)^{|n|-1}q'(\theta)e^{in\theta}\,d\theta\Big|\\
\le \frac{|n|-1}{|n|}\int_{\alpha}^{\beta}q(\theta)^{|n|-2}|q'(\theta)|^2\,d\theta+
\frac1{|n|}\int_{\alpha}^{\beta}q(\theta)^{|n|-1}|q''(\theta)|\,d\theta\\
\le 2p^2\int_0^{\tau/2}\Big(1-\Bigl[\frac t2\Bigr]^p\Big)^{|n|-2}t^{2p-2}\,dt
+\frac{2p^2}{|n|}\int_0^{\tau/2}\Big(1-\Bigl[\frac t2\Bigr]^p\Big)^{|n|-1}t^{p-2}\,dt\\
\le C_p\Bigl[\int_0^{\tau/4}\big(1-t^p\big)^{|n|-2}t^{2p-2}\,dt
+\frac1{|n|}\int_0^{\tau/4}\big(1-t^p\big)^{|n|-1}t^{p-2}\,dt\Bigr],
\end{multline*}
and, hence,
\begin{eqnarray*}
\Big|\int_\alpha^\beta(1-q(\theta)^{|n|})e^{in\theta}\,d\theta\Big|
&\leq &
C_{1,p}\tau\max_{0\le t\le 1}\Bigl\{\bigl(1-t^p\bigr)^{|n|-2}t^{2p-2}+\frac{1}{
|n|}\bigl(1-t^p\bigr)^{|n|-1}t^{p-2}\Bigr\}\\
&\leq &C_{2,p}\tau |n|^{-2( 1-\frac{1}{p}) }.
\end{eqnarray*}

On the other hand, we have
$$
\frac 1{2\pi}\int_I h(r\xi)\,|d\xi|=\frac 1{2\pi}\int_\alpha^\beta h(rq(\theta)e^{i\theta})\,d\theta+
\frac 1{2\pi}\int_\alpha^\beta\left[h(re^{i\theta})-h(rq(\theta)e^{i\theta})\right]\,d\theta.
$$
By \eqref{xf2}, we obtain 
\begin{eqnarray*}
\Big|\frac 1{2\pi}\int_\alpha^\beta\big[h(re^{i\theta})-h(rq(\theta)e^{i\theta})\big]\,d\theta\Big|
&\le &\frac 1{2\pi }\sum_{n\in\mathbb Z}|a_n|\Big|\int_\alpha^\beta r^{|n|}(1-q(\theta)^{|n|})e^{in\theta}\,d\theta\Big| \\
&\le &C_{3,p}\tau\sum_{n\in\mathbb Z}|a_n| (|n|+1)^{-2(1-\frac{1}{p})} \\
& \leq &C_{4,p}\tau \sum_{n\in\mathbb Z}\Lambda\Big(\frac{1}{\max(|n|,1)}\Big) 
(|n|+1)^{-2(1-\frac{1}{p}) }.
\end{eqnarray*}

Therefore, if $t\mapsto t^\alpha\Lambda(t)$ increase, and 
\begin{equation}
\alpha+\frac 2p<1,
\end{equation}
then
$$
\Big|\frac 1{2\pi}\int_\alpha^\beta\big[h(re^{i\theta})-h(rq(\theta)e^{i\theta})\big]\,d\theta\Big|\le C_{5,p}\tau.
$$

Since $\Lambda(x^p)\le C_p\Lambda(x)$, we obtain
\begin{eqnarray*}
\Big|\frac 1{2\pi}\int_\alpha^\beta h(rq(\theta)e^{i\theta})\,d\theta\Big|
&\le& C\int_\alpha^\beta\Lambda(1-q(\theta))\,d\theta \\
&\le& C\int_\alpha^\beta\Lambda\bigl(\frac{t(\theta)}2\bigr)\,d\theta\\
&\le& C_1\int_0^{\tau/4}\Lambda(t)\,dt\\
&=& C_1\int_0^{\tau/4}t^{-\alpha}t^\alpha\Lambda(t)\,dt\\
&\le& C_2\tau^\alpha\Lambda(\tau)\int_0^{\tau/4}t^{-\alpha}\,dt\\
&=& C_3\tau\Lambda(\tau).
\end{eqnarray*}
Hence,
\begin{equation*}
\mu_r(I)\le C|I|\Lambda(|I|)
\end{equation*}
for some $C$ independent of $I$.

Given $r\in(0,1)$, we define $h_r(z)=h(rz)$. The $h_{r}$ is the Poisson integral of $d\mu_r=h_r(e^{i\theta})\,d\theta$:
$$
h_r(z)=\int_{\TT}\frac{1-|z|^2}{|e^{i\theta}-z|^2}\,d\mu_r(\theta)
$$
The set $\{\mu_r:r\in(0,1)\}$ is a uniformly $\Lambda$-bounded family of premeasures. Using a Helly-type selection theorem \cite[Theorem~1, p.~204]{BK1}, we can find a sequence of premeasures $\mu_{r_n}\in\LB$ converging weakly to a $\Lambda$-bounded premeasure $\mu$ as $n\to\infty$, 
$\lim_{n\to\infty}r_n=1$. 
Then
\begin{equation*}
\mu(I)\le C|I|\Lambda(|I|)
\end{equation*}
for every arc $I$, and
$$
h_{r_n}(z)=-\int_0^{2\pi}\frac{\partial}{\partial\theta}\Big(\frac{1-|z|^2}{|e^{i\theta}-z|^2}\Big)
\overset{\curlywedge}{\mu}_n(\theta)\,d\theta.
$$
Passing to the limit we conclude that
$$
h(z)=\int_{\TT}\frac{1-|z|^2}{|e^{i\theta}-z|^2}\,d\mu(\theta).
$$
\end{proof}

\section{Cyclic vectors}

Given a $\Lambda$-bounded premeasure $\mu$, we consider the corresponding analytic fuction
\begin{equation}
\label{xdf1}
f_\mu(z)=\exp \int_0^{2\pi }\dfrac{e^{i\theta }+z}{e^{i\theta }-z}\,d\mu (\theta).
\end{equation}
If $\tilde\mu$ is a positive singular measure on the circle $\TT$, we denote by $S_{\tilde\mu}$ the associated singular inner function. Notice that in this case $\mu=\tilde\mu(\mathbb T)m-\tilde\mu$
is a premeasure, and we have 
$S_{\tilde\mu}= f_\mu/S_{\tilde\mu}(0)$; $m$ is (normalized) Lebesgue measure.

Let $f$ be a zero-free function in $\mathcal{A}_{\Lambda}^{-\infty }$ such that $f(0)=1$. According to Theorem \ref{harmonic}, there is a premeasure $\mu_f\in \LB$ such that
$$
f(z)=\exp\int_{0}^{2\pi}\frac{e^{i\theta}+z}{e^{i\theta}-z}\,d\mu_f(\theta).
$$
The following result follows immediately from Theorem \ref{theo app }. 

\begin{theo}\label{uneimplication}
Let $f \in \mathcal{A}_{\Lambda}^{-\infty }$ be a zero-free function such that $f(0)=1$.
If $(\mu_{f})_{s}\equiv 0$, then $f$ is cyclic in $\mathcal{A}_{\Lambda}^{-\infty }$.
\end{theo}

\begin{proof}
Suppose that $(\mu_{f})_s\equiv 0$. By theorem \ref{theo app }, 
$\mu_f$ is $\Lambda$-absolutely continuous. Let $\{\mu_n\}_{n\ge 1}$ 
be a sequence of $\Lambda$-bounded premeasures from Definition~\ref{def1}.
We set
$$
g_n(z)=\exp\int_{0}^{2\pi}\frac{e^{i\theta}+z}{e^{i\theta}-z}d\mu_{n}(\theta),
\qquad z\in\DD.
$$
By Proposition~\ref{har-rep}, $g_n\in \mathcal{A}_{\Lambda}^{-\infty }$, and
\begin{eqnarray*}
f(z)g_n(z)&=&\exp\int_0^{2\pi}\frac{e^{i\theta}+z}{e^{i\theta}-z}\,d(\mu_f+\mu_n)(\theta)\\
&=&\exp\Bigl[-\int_0^{2\pi}\frac{\partial}{\partial\theta}\Big(\frac{e^{i\theta}+z}{e^{i\theta}-z}\Big)[\overset{\curlywedge}{\mu_{n}}( \theta)-\overset{\curlywedge}{\mu}(\theta)]\,
d\theta\Bigr]\\
&=&\exp\Bigl[-\int_0^{2\pi}\frac{\partial}{\partial\theta}\Big(\frac{e^{i\theta}+z}{e^{i\theta}-z}\Big)\big[\mu(I_{\theta})+\mu_{n}(I_{\theta})\big]\,d\theta\Bigl].
\end{eqnarray*}
Again by Definition~\ref{def1}, we obtain that $f(z)g_n(z)\to 1$ uniformly on compact subsets of unit disk $\DD$. This yields that $fg_n\to 1$ in 
$\mathcal{A}_\Lambda^{-\infty}$ as $n\to\infty$.
\end{proof}

From now on, we deal with the statements converse to Theorem \ref{uneimplication}.
We'll establish two results valid for different growth ranges of the majorant $\Lambda$.
More precisely, we consider the following growth and regularity assumptions:
\begin{gather}
\tag{C1} \text{for every\ }c>0, \text{\ the function\ }x\mapsto\exp\bigl[c\Lambda(1/x)\bigr]
\text{\ is concave for large\ }x,\\
\tag{C2} \lim_{t\to 0}\frac{\Lambda(t)}{\log(1/t)}=\infty.
\end{gather}
Examples of majorants $\Lambda$ satisfying condition (C1) include
$$
(\log (1/x))^p,\quad 0<p<1,\quad \text{and}\quad \log(\log (1/x)),\quad x\to 0.
$$
Examples of majorants $\Lambda$ satisfying condition (C2) include
$$
(\log (1/x))^p,\quad p>1.
$$
Thus, we consider majorants which grow less rapidly than 
the Korenblum majorant ($\Lambda(x)=\log (1/x)$) in Case 1 or
more rapidly than 
the Korenblum majorant in Case 2.

\subsection{Weights $\Lambda$ satisfying condition (C1)}
We start with the following observation:
$$
\Lambda(t)=o(\log 1/t),\qquad t\to 0.
$$
Next we pass to some notations and auxiliary lemmas. Given a function $f$ in 
$L^1(\TT)$, we denote by $P[f]$ its Poisson transform,
$$
P[f](z)=\frac 1{2\pi}\int_0^{2\pi} \frac{1-|z|^{2}}{|e^{i\theta}-z |}f(e^{i\theta})\,d\theta, \qquad z\in \DD.
$$
Denote by $A(\DD)$ the disk-algebra, i.e., the algebra of functions continuous on the closed unit disk and holomorphic in $\DD$. 
A positive continuous increasing function $\omega$ on $[0,\infty)$ is said to be a modulus of continuity if $\omega(0)=0$, $t\mapsto\omega(t)/t$ decreases near $0$, 
and $\lim_{t\to 0}\omega(t)/t=\infty$.
Given a modulus of continuity $\omega$, 
we consider the Lipschitz space $\mathrm{Lip}_\omega(\TT)$ defined by 
$$
\Lip_\omega(\TT)=\{f\in C(\TT):|f(\xi)-f(\zeta)|\le C(f)\omega(|\xi-\zeta|)\}.
$$

Since the function $t\mapsto \exp[2\Lambda(1/t)]$ is concave for large $t$, and 
$\Lambda(t)=o(\log(1/t))$, $t\to 0$, we can apply a result of Kellay 
\cite[Lemma~3.1]{KK}, to get a non-negative summable function 
$\Omega_\Lambda$ on $[0,1]$ such that 
$$ 
e^{2\Lambda(\frac{1}{n+1})}-e^{2\Lambda(\frac{1}{n})}\asymp\int_{1-\frac{1}{n}}^{1}\Omega_{\Lambda}(t)dt,\qquad n\ge 1.
$$

Next we consider the Hilbert space $L_{\Omega_\Lambda}^{2}(\TT)$
of the functions $f\in L^2(\TT)$ such that 
$$
\|f\|_{\Omega_\Lambda}^2=|P[f](0)|^{2}+\int_{\DD}\frac{P[|f|^{2}](z)-|P[f](z)|^{2}}{1-|z|^{2}}\Omega_{\Lambda}(|z|)\,dA(z)<\infty,
$$
where $dA$ denote the normalized area measure. 
We need the following lemma.

\begin{lem}\label{lem ind} Under our conditions on $\Lambda$ and $\Omega_\Lambda$, we have  
\begin{enumerate}
\item $\|f\|_{\Omega_\Lambda}^2\asymp \sum_{n\in \ZZ}|\hat f(n)|^{2}
e^{2\Lambda(1/n)}$,\quad $f\in L_{\Omega_\Lambda}^{2}(\TT)$,
\item the functions $\exp(-c\Lambda(t))$ are moduli of continuity for $c>0$,
\item for some positive $a$, the function $\rho(t)=\exp(-\frac{3}{2a}\Lambda(t))$ satisfies the property
$$
\Lip_\rho(\TT)\subset L_{\Omega_\Lambda}^2(\TT).
$$
\end{enumerate}
\end{lem}

For the first statement see \cite[Lemma~6.1]{BFK} (where it is attributed to Aleman 
\cite{AL}); the second statement is \cite[Lemma~8.4]{BFK}; the third statement 
follows from \cite[Lemmas~6.2 and 6.3]{BFK}.

Recall that 
$$
\mathcal{A}_{\Lambda }^{-1}=\{f\in\Hol(\DD):|f(z)|\le C(f)\exp (\Lambda(1-|z|))\}.
$$

\begin{lem}\label{duality}
Under our conditions on $\Lambda$, there exists a positive number $c$ such that 
$$
P_+\Lip_{e^{-c\Lambda}}(  \TT)
\subset(\mathcal{A}_{\Lambda }^{-1})^{*}
$$
via the Cauchy duality  
$$
\langle f,g\rangle=\sum_{n\ge 0}a_n\overline{\widehat{g}(n)},
$$
where $f(z) =\sum_{n\ge 0}a_n z^n \in\mathcal{A}_{\Lambda}^{-1}$,  
$g\in\Lip_{e^{-c\Lambda}}(  \TT)$, and $P_+$ is  the orthogonal projector from $ \IL(\TT)$ onto $\hH(\DD)$.
\end{lem}

\begin{proof}
Denote 
$$
L_{\Lambda}^2(\DD)=\Big\{f\in\Hol(\DD): 
\int_{\DD}|f(z)|^2|\Lambda'(1-|z|)|e^{-2\Lambda(1-|z|)}\,dA(z)<+\infty\Big\},
$$
and
$$
\mathcal{B}_\Lambda^2=\big\{f(z)=\sum_{n\ge 0} a_nz^n: |a_0|^2+\sum_{n>0}
|a_n|^2e^{-2\Lambda(1/n)}<\infty\big\}.
$$
Let us prove that 
\begin{equation}
\label{xf5}
L_\Lambda^2(\DD)=\mathcal{B}_\Lambda^2.
\end{equation}
To verify this equality, it suffices sufficient to check that 
$$
e^{-2\Lambda(1/n)} \asymp \int_0^1 r^{2n+1}|\Lambda'(1-r)|e^{-2\Lambda(1-r)}\,dr.
$$
In fact, 
$$
\int_{1-1/n}^{1}r^{2n+1}|\Lambda^{'}(1-r)|e^{-2\Lambda(1-r)}\,dr\asymp \int_{1-1/n}^{1}|\Lambda^{'}(1-r)|e^{-2\Lambda(1-r)}\,dr \asymp e^{-2\Lambda(\frac{1}{n})},
\qquad n\ge 1.
$$
On the other hand,
\begin{eqnarray*}
\int_0^{1-1/n}r^{2n+1}|\Lambda'(1-r)|e^{-2\Lambda(1-r)}\,dr& = & 
-\int_0^{1-1/n}r^{2n+1}\,de^{-2\Lambda(1-r)} \\
&\asymp & -e^{-2\Lambda(1/n)}+(2n+1)\int_0^{1-1/n} r^{2n}e^{-2\Lambda(1-r)}dr \\
&\asymp& n\sum_{k=1}^{n}e^{-2n/k}e^{-2\Lambda(1/k)}\frac 1{k^2}.
\end{eqnarray*}
Since the function $\exp\bigl[2\Lambda(1/x)\bigr]$ is concave, 
we have $e^{2\Lambda(1/k)}\ge \frac kn e^{2\Lambda(1/n)}$, and hence,
$$
e^{-2\Lambda(1/k)}\le \frac nk e^{-2\Lambda(1/n)}.
$$
Therefore,
$$
\int_0^{1-1/n}r^{2n+1}|\Lambda'(1-r)|e^{-2\Lambda(1-r)}\,dr \le 
Cn^2e^{-2\Lambda(1/n)}\sum_{k=1}^{n}e^{-2n/k}\frac 1{k^3}\asymp 
e^{-2\Lambda(1/n)},
$$
and \eqref{xf5} follows.

Since $\mathcal{A}_\Lambda^{-1}\subset L_\Lambda^2(\DD)$, we have 
$(\mathcal{B}_\Lambda^2)^*\subset(\mathcal{A}_\Lambda^{-1})^*$.
By Lemma~\ref{lem ind}, we have $P_+\Lip_\rho(\TT)
\subset 
(\mathcal{B}_\Lambda^2)^*$. Thus, 
$$
P_+\Lip_\rho(\TT)
\subset(\mathcal{A}_\Lambda^{-1})^*.
$$
\end{proof}

\begin{lem}\label{limit inductive}
Let $f\in\mathcal{A}_\Lambda^{-n}$ for some $n>0$. The function $f$ is cyclic in 
$\mathcal{A}_\Lambda^{-\infty}$ if and only if there exists $m>n$ such that $f$ is cyclic in $\mathcal{A}_\Lambda^{-m}$.
\end{lem}

\begin{proof} Notice that the space $\mathcal{A}_{\Lambda}^{-\infty }$ is endowed with the inductive limit topology induced by the spaces $\mathcal{A}_{\Lambda}^{-N }$. A sequence $\{f_n\}_n\in\mathcal{A}_\Lambda^{-\infty }$ converges to 
$g\in\mathcal{A}_\Lambda^{-\infty}$ if and only if there exists $N>0$ such that 
all $f_n$ and $g$ belong to $\mathcal{A}_\Lambda^{-N}$, and 
$\lim_{n\to+\infty}\|f_n-g\|_{\mathcal{A}_\Lambda^{-N}}=0$. The statement of the lemma follows.
\end{proof}

\begin{theo}\label{theo}
Let $\mu \in \LB$, and let the majorant $\Lambda$ satisfy condition {\rm(C1)}.
Then the function $f_\mu$ is cyclic in $\mathcal{A}_\Lambda^{-\infty}$ if and only if 
$\mu_s\equiv 0$.
\end{theo}

\begin{proof}
Suppose that the $\Lambda$-singular part $\mu_s$ of $\mu$ is non-trivial.
There exists a $\Lambda$-Carleson set $F\subset \TT$ such that 
$-\infty <\mu_s(F)<0$. We set $\nu=-\mu_s{\bigm|}F$. By a theorem of Shirokov 
\cite[Theorem 9, pp.137,139]{shi}, there exists an outer function $\varphi $ such that
$$
\varphi\in \Lip_\rho(\TT) \cap \kH(\DD),\qquad 
\varphi S_\nu\in \Lip_\rho(\TT)\cap \kH(\DD),
$$ 
and the zero set of the function $\varphi$ coincides with $F$. Next, for 
$\xi,\theta\in[0,2\pi]$ we have
\begin{eqnarray*}
|\varphi \overline{S_{\nu}}(e^{i\xi})-\varphi \overline{S_{\nu}}(e^{i\theta})|&=&|\varphi(e^{i\xi}) S_{\nu}(e^{i\theta})-\varphi(e^{i\theta}) S_{\nu}(e^{i\xi})|\\
&\leq&|\big(\varphi(e^{i\xi})-\varphi (e^{i\theta})\big)S_{\nu}(e^{i\theta})|+|\big(\varphi(e^{i\theta})-\varphi (e^{i\xi})\big)S_{\nu}(e^{i\xi})|\\
&+&|(\varphi S_{\nu})(e^{i\theta})-(\varphi S_{\nu})(e^{i\xi})|,
\end{eqnarray*}
and hence, 
$$
\varphi \overline{S_\nu}\in \Lip_\rho(\TT).
$$

Set $g=P_+\big(\overline{z\varphi }S_\nu\big)$. 
Since $\varphi\overline{S_\nu}\in \textrm{Lip}_\rho(\TT)$, we have 
$g\in 
(\mathcal{A}_\Lambda^{-1})^*$.
Consider the following linear functional on $\mathcal{A}_{\Lambda}^{-1}$:
$$
L_{g}(f) =\langle f,g\rangle=\sum_{n\ge 0}a_{n}\overline{\widehat{g}(n) },\qquad  
f(z)=\sum_{n\ge 0}a_{n}z^{n}\in\mathcal{A}_\Lambda^{-1}. 
$$

Suppose that $L_g=0$. Then, for every $n\ge 0$ we have
\begin{eqnarray*}
0 &=& L_{g}( z^{n})  \\
  &=& \int_0^{2\pi} e^{in\theta }\overline{g(e^{i\theta }) }\frac{d\theta }{2\pi } \\
  &=& \int_0^{2\pi} e^{i( n+1) \theta }\frac{\varphi ( e^{i\theta
}) }{S_{\nu}(e^{i\theta }) }\frac{d\theta }{2\pi }.
\end{eqnarray*}
We conclude that $\varphi/S_\nu\in \kH(\DD)$, which is impossible. 
Thus, $L_g\neq 0$. 

On the other hand we have, for every $n\ge 0$,
\begin{eqnarray*}
L_g(z^nS_\nu) &=& \int_0^{2\pi} e^{in\theta })S_{\nu}( e^{i\theta }) \overline{g( e^{i\theta }) }\frac{
d\theta }{2\pi }\\
&=& \int_0^{2\pi} e^{i n\theta }S_\nu( e^{i\theta }) \overline{g( e^{i\theta }) }
\frac{d\theta }{2\pi }\\
&=&\int_0^{2\pi} e^{i( n+1) \theta }\varphi ( e^{i\theta })
\frac{d\theta }{2\pi }\\&=&0.
\end{eqnarray*}
Thus,
$g\perp \left[f_{\mu}\right]_{\mathcal{A}_{\Lambda }^{-1 }}$ which 
implies that the function $f_\mu$ is not cyclic in ${\mathcal{A}_\Lambda^{-1}}$. 
By Lemma~\ref{limit inductive}, $f_\mu$ is not cyclic in
$\mathcal{A}_\Lambda^{-\infty }$.
\end{proof}

\subsection{Weights $\Lambda$ satisfying condition (C2)}

We start with an elementary consequence of the Cauchy formula.

\begin{lem}\label{dualite-2} Let $f(z)=\underset{n\geq0}\sum a_{n}z^{n}$ be an analytic function in $\DD$. If $f \in \mathcal{A}_{\Lambda}^{-\infty }$, then there exists $C>0$ such that
$$
|a_{n}|= O(\exp[C\Lambda(\frac{1}{n})]) \quad \text{as} \quad n\to +\infty.
$$
\end{lem}

\begin{theo}\label{the cy 2} 
Let $\mu \in \LB$, and let the majorant $\Lambda$ satisfy condition {\rm(C2)}.
Then the function $f_\mu$ is cyclic in $\mathcal{A}_\Lambda^{-\infty}$ if and only if 
$\mu_s\equiv 0$.
\end{theo}

\begin{proof}
We define
$$
\mathcal{A}_\Lambda^{\infty}=\bigcap_{c<\infty}\Big\{g\in \Hol(\DD)\cap C^{\infty}(\bar\DD): 
|\widehat{f}(n)|= O( \exp[-c\Lambda(\frac{1}{n})])\Big\},
$$
and, using Lemma \ref{dualite-2}, we obtain that ${\mathcal{A}_{\Lambda}^{\infty}} \subset({\mathcal{A}_{\Lambda}^{-\infty}})^{*}$ via the Cauchy duality
$$
\left\langle f, g\right\rangle = \sum_{n\geq 0}\widehat{f}(n)\overline{\widehat{g}(n)} =\lim_{r\to 1} \int^{2\pi}_{0}f(r\xi)\overline{g(\xi)}d\xi, \qquad f\in{\mathcal{A}_\Lambda^{-\infty}} ,\quad  g\in{\mathcal{A}_\Lambda^{\infty}}.
$$

Suppose that the $\Lambda$-singular part $\mu_s$ of $\mu$ is nonzero. Then there exists a $\Lambda$-Carleson set $F\subset \TT$ such that $-\infty <\mu_s(F)<0$. 
We set $\sigma=\mu_s{\bigm|}F$. 
By a theorem of Bourhim, El-Fallah, and Kellay \cite[Theorem 5.3]{BFK} 
(extending a result of Taylor and Williams), there exist an outer function $\varphi\in {\mathcal{A}_{\Lambda}^{\infty}}$ such that the zero set of $\varphi$ and of all its derivatives coincides exactly with the set $F$, a function $\widetilde{\Lambda}$ such that 
\begin{equation}
\label{xd4}
\Lambda(t)=o(\widetilde{\Lambda}(t)), \qquad t\to0,
\end{equation}
and a positive constant $B$ such that
\begin{equation}
\label{xe1}
|\varphi^{(n)}(z)|\le n! B^{n} e^{\widetilde{\Lambda}^{*}(n)},\qquad n\geq 0 ,\; z\in \DD,
\end{equation}
where $\widetilde{\Lambda}^*(n)= \sup_{x>0} \bigl \{nx-\widetilde{\Lambda}(e^{-x/2})\bigr\}$.

We set
$$
\Psi=\varphi\overline{S_\sigma}.
$$
For  some positive $D$ we have
\begin{equation}
\label{xe2}
|S_{\sigma}^{(n)}(z)|\le\frac{D^{n}n!}{\dist(z,F)^{2n}},\qquad z\in \mathbb{D},\,n\ge 0.
\end{equation}
By the Taylor formula, for every $n,k\ge 0$, we have
\begin{equation}
\label{xe3}
|\varphi^{(n)}(z)|\le \frac{1}{k!}\dist(z,F)^{k}\max_{w\in \DD}|\varphi^{(n+k)}(w)|, \qquad z\in \mathbb{D}.
\end{equation}
Next, integrating by parts, for every $n\not=0$, $k\ge 0$ we obtain 
$$
|\widehat{\Psi}(n)|=\bigl|\widehat{(\varphi\overline{S_\sigma})}(n)\bigr|
=\frac 1{2\pi}\Bigl|\int^{2\pi}_0 \frac{(\varphi \overline{S_\sigma})^{(k)}(e^{it})}{n^k} 
e^{-int} \,dt\Bigr|.
$$

Applying the Leibniz formula and estimates \eqref{xe1}--\eqref{xe3}, we obtain for $n\ge1$ that	
\begin{eqnarray*}
\bigl|\widehat{\Psi}(n)\bigr| &\le&
\inf_{k\ge 0}\Bigl\{\frac{1}{n^{k}}\max_{t\in[0,2\pi]}
\left|(\varphi\overline{S_{\sigma}})^{(k)}(e^{it})\right|\Bigr\}\\
&\le&\inf_{k\ge 0}\Bigl\{\frac{1}{n^{k}}\sum_{s=0}^k C^s_k
\max_{t\in[0,2\pi]}|S_{\sigma}^{(s)}(e^{it})|
\max_{t\in[0,2\pi]}|\varphi^{(k-s)}(e^{it})|\Bigr\}\\
&\le&\inf_{k\ge 0}\Bigl\{\frac{1}{n^{k}}
\sum_{s=0}^k C^s_k D^{s}s! \frac{1}{(2s)!} 
(k+s)! B^{k+s} e^{\widetilde{\Lambda}^{*}(k+s)}
\Bigr\}\\
&\le& \inf_{k\ge 0}\Bigl\{e^{\widetilde{\Lambda}^{*}(2k)}
\Bigl(\frac{B^2D}{n}\Bigr)^{k}
\sum_{s=0}^k  \frac{(k+s)!k!}{(2s)!(k-s)!} \Bigr\}\\
&\le& \inf_{k\ge 0}\Bigl\{k!e^{\widetilde{\Lambda}^{*}(2k)}
\Bigl(\frac{4B^2D}{n}\Bigr)^{k}\Bigr\}\\
&\le& \inf_{k\ge 0}k!\Bigl\{\Bigl(\frac{4B^2D}{n}\Bigr)^{k}  
\sup_{0<t<1}\Bigl\{e^{-\widetilde{\Lambda}(t^{1/4})}t^{-k}\Bigr\}\Bigr\}.
\end{eqnarray*}

By property \eqref{xd4}, for every $C>0$ there exists a positive number $K$ such that 
$$
e^{-\widetilde{\Lambda}(t^{1/4})}\le Ke^{-\Lambda(Ct)}, 
\qquad t\in(0,1).
$$

We take $C=\frac{1}{8B^2D}$, and obtain for $n\not=0$ that 
\begin{eqnarray*}
\bigl|\widehat{\Psi}(n)\bigr| 
&\le& K\inf_{k\ge 0}\Bigl\{\Big(\frac{4B^2D}{n}\Big)^k k! \sup_{0<t<1}\frac{e^{-\Lambda(Ct)}}{t^{k}}\Bigr\}\\
&\le&
K_1\inf_{k\ge 0}\Bigl\{(2n)^{-k} k! \sup_{0<t<1}\frac{e^{-\Lambda(t)}}{t^{k}}\Bigr\}.
\end{eqnarray*}
Finally, using \cite[Lemma~6.5]{KHR} (see also \cite[Lemma~8.3]{BFK}), we get 
$$
|\widehat{\Psi}(n)|=O\bigl(e^{-\Lambda(1/n)}\bigr),\qquad |n|\to\infty.
$$

Thus, the function 
$g=P_+\big(\overline{z\varphi}S_{\sigma}\big)$ belongs to 
$({\mathcal{A}_{\Lambda}^{-1}})^*$. Now we obtain that 
$f_{\mu}$ is not cyclic using the same argument as that 
at the end of Case~1. This 
concludes the proof 
of the theorem.
\end{proof}

Theorems~\ref{theo} and \ref{the cy 2} together give a positive answer 
to a conjecture by Deninger \cite[Conjecture 42]{CD}.

We complete this section by two examples that show how the cyclicity 
property of a fixed function changes in a scale of 
$\mathcal{A}_{\Lambda}^{-\infty}$ spaces.

\begin{example} \label{ex1} Let $\Lambda_\alpha(x)=(\log (1/x))^\alpha$, $0<\alpha<1$, and let $0<\alpha_0<1$.
There exists a singular inner function 
$S_\mu$ such that
$$
S_\mu\text{\ is cyclic in\ }\mathcal{A}_{\Lambda_\alpha}^{-\infty} \iff \alpha>\alpha_0.
$$
\end{example}

{\it Construction.} We start by defining a Cantor type set 
and the corresponding canonical measure. Let $\{m_k\}_{k\ge 1}$ be a sequence of natural numbers. Set $M_k=\sum_{1\le s\le k}m_s$, and assume that 
\begin{equation}
\label{1k3}
M_k\asymp m_k, \qquad k\to\infty.
\end{equation}
Consider the following iterative procedure. Set $\mathcal I_0=[0,1]$. On the step $n\ge 1$ the set $\mathcal I_{n-1}$ 
consist of several intervals $I$. We divide each $I$ into 
$2^{m_n+1}$ equal subintervals and replace it by the union of 
every second interval in this division. The union of all such groups 
is $\mathcal I_n$. Correspondingly, $\mathcal I_n$ consists 
of $2^{M_n}$ intervals; each of them is of length 
$2^{-n-M_n}$. Next, we consider the probabilistic measure $\mu_n$ equidistributed on $\mathcal I_n$. 
Finally, we set $E=\cap_{n\ge 1}\mathcal I_n$, 
and define by $\mu$ the weak limit of the measures $\mu_n$.

Now we estimate the $\Lambda_\alpha$-entropy of $E$:
$$
Entr_{\Lambda_\alpha}(\mathcal I_n)\asymp\sum_{1\le k\le n}
2^{M_k}\cdot 2^{-k-M_k}\cdot
\Lambda_\alpha(2^{-k-M_k})\asymp \sum_{1\le k\le n}2^{-k}\cdot m_k^\alpha,
\qquad n\to\infty.
$$
Thus, if 
\begin{equation}
\label{1k1}
\sum_{n\ge 1}2^{-n}\cdot m_n^{\alpha_0}<\infty,
\end{equation}
then $Entr_{\Lambda_{\alpha_0}}(E)<\infty$.
By Theorem~\ref{theo}, $S_\mu$ is not cyclic in 
$\mathcal{A}_{\Lambda_\alpha}^{-\infty}$ for $\alpha\le\alpha_0$.

Next we estimate the modulus of continuity of the measure $\mu$,
$$
\omega_\mu(t)=\sup_{|I|=t}\mu(I).
$$
Assume that 
$$
A_{j+1}=2^{-{(j+1)}-M_{j+1}}\le |I|< A_j=2^{-j-M_j},
$$
and that $I$ intersects with one of the intervals $I_j$
that constitute $\mathcal I_j$. 
Then
$$
\mu(I)\le 4\frac{|I|}{A_j}\mu(I_j)=4|I|2^{j+M_j}2^{-M_j}=
4|I|2^j.
$$ 
Thus, if 
\begin{equation}
\label{1k2}
2^j\le C (\log (1/A_j))^\alpha
\asymp m_j^\alpha, \qquad j\ge 1,\,\alpha_0<\alpha<1,
\end{equation}
then 
$$
\omega_\mu(t)\le Ct(\log (1/t))^\alpha.
$$
By \cite[Corollary B]{BBC}, we have $\mu(F)=0$ for any 
$\Lambda_\alpha$-Carleson set $F$, $\alpha_0<\alpha<1$.
Again by Theorem~\ref{theo}, $S_\mu$ is cyclic in 
$\mathcal{A}_{\Lambda_\alpha}^{-\infty}$ for $\alpha>\alpha_0$.
It remains to fix $\{m_k\}_{k\ge 1}$ satisfying 
\eqref{1k3}--\eqref{1k2}. The choice $m_k=
2^{k/\alpha_0}k^{-2/\alpha_0}$ works.\hfill $\Box$
\medskip

Of course, instead of Theorem~\ref{theo} we could use here 
\cite[Theorem~7.1]{BFK}.

\begin{example} Let $\Lambda_\alpha(x)=(\log (1/x))^\alpha$, $0<\alpha<1$, and let $0<\alpha_0<1$.
There exists a premeasure $\mu$ such that $\mu_s$ is infinite, 
$$
f_\mu\text{\ is cyclic in\ }\mathcal{A}_{\Lambda_\alpha}^{-\infty} \iff \alpha>\alpha_0,
$$
where $f_\mu$ is defined by \eqref{xdf1}.
\end{example}

It looks like the subspaces 
$[f_\mu]_{\mathcal{A}_{\Lambda_\alpha}^{-\infty}}$, 
$\alpha\le\alpha_0$,
contain no nonzero Nevanlinna class functions.
For a detailed discussion on Nevanlinna class generated 
invariant subspaces in the Bergman space (and in the Korenblum space) 
see \cite{HKZ1}.

For $\alpha\le\alpha_0$, instead of Theorem~\ref{theo} we could once again use here \cite[Theorem~7.1]{BFK}.
\medskip

{\it Construction.} We use the measure $\mu$ constructed in 
Example~\ref{ex1}.

Choose a decreasing sequence $u_k$ of positive numbers such that
$$
\sum_{k\ge 1}u_k=1,\qquad \sum_{k\ge 1}v_k=+\infty,
$$
where $v_k=u_k\log\log(1/u_k)>0$, $k\ge 1$.

Given a Borel set $B\subset B^0=[0,1]$, denote
$$
B_k=\{u_kt+\sum_{j=1}^{k-1}u_j:t\in B\}\subset[0,1],
$$
and define measures $\nu_k$ supported by $B^0_k$ by 
$$
\nu_k(B_k)=\frac{v_k}{u_k}m(B_k)-v_k\mu(B),
$$
where $m(B_k)$ is Lebesgue measure of $B_k$.

We set 
$$
\nu=\sum_{k\ge 1}\nu_k.
$$
Then $\nu(B^0_k)=\nu_k(B^0_k)=0$, $k\ge 1$, and $\nu$ is a premeasure.

Since
$$
v_k\le C(\alpha)u_k\Lambda_\alpha(u_k),\qquad 0<\alpha<1,
$$
$\nu$ is a $\Lambda_\alpha$-bounded premeasure for $\alpha\in(0,1)$.

Furthermore, as above, by Theorem~\ref{theo}, $f_\nu$ is not cyclic in 
$\mathcal{A}_{\Lambda_\alpha}^{-\infty}$ for $\alpha\le\alpha_0$.

Next, we estimate  
$$
\omega_\nu(t)=\sup_{|I|=t}|\nu(I)|.
$$
As in Example~\ref{ex1}, if $j,k\ge 1$ and
$$
u_kA_{j+1}\le |I|< u_kA_j,
$$ 
then
\begin{equation}
\frac{|\nu(I)|}{|I|}\le C\cdot 2^j \cdot \frac{v_k}{u_k}.
\label{eqa}
\end{equation}
Now we verify that
\begin{equation}
\omega_\nu(t)\le Ct(\log (1/t))^\alpha, \qquad \alpha_0<\alpha<1.
\label{eqb}
\end{equation}
Fix $\alpha\in(\alpha_0,1)$, and use that
$$
\bigl(\log\frac{1}{A_j}\bigr)^\alpha\ge C\cdot 2^{(1+\varepsilon)j},
\qquad j\ge 1,
$$
for some $C,\varepsilon>0$. By \eqref{eqa}, it remains to check that
$$
2^j\log\log\frac{1}{u_k}\le C\bigl(2^{(1+\varepsilon)j}+
\bigl(\log\frac{1}{u_k}\bigr)^\alpha\bigr).
$$
Indeed, if 
$$
\log\log\frac{1}{u_k}>2^{\varepsilon j},
$$
then
$$
C\bigl(\log\frac{1}{u_k}\bigr)^\alpha>2^j\log\log\frac{1}{u_k}.
$$

Finally, we fix $\alpha\in(\alpha_0,1)$ and a 
$\Lambda_\alpha$-Carleson set $F$. We have  
$$
\mathbb T\setminus F=\sqcup_s L^*_s
$$
for some intervals $L^*_s$.
By \cite[Theorem B]{BBC}, there exist disjoint intervals $L_{n,s}$ such that 
$$
F\subset \sqcup_s L_{n,s},\qquad \sum_s |L_{n,s}|\Lambda_\alpha(|L_{n,s}|)<\frac 1n,\qquad n\ge 1.
$$
Then by \eqref{eqb},
$$
\sum_s|\nu(L_{n,s})|<\frac cn.
$$
Set
$$
\mathbb T\setminus \sqcup_s L_{n,s}=\sqcup_s L^*_{n,s}.
$$
Then
$$
\bigl|\sum_s \nu(L^*_{n,s}) \bigr|<\frac cn.
$$
Since $F$ is $\Lambda_\alpha$-Carleson, we have
$$
\sum_s |L^*_{s}|\Lambda_\alpha(|L^*_{s}|)<\infty,
$$
and hence,
$$
\sum_s \nu(L^*_{n,s})\to \sum_s \nu(L^*_{s})
$$
as $n\to\infty$.
Thus, 
$$
\sum_s \nu(L^*_{s})=0,
$$
and hence, $\nu(F)=0$. 
Again by Theorem~\ref{theo}, $f_\nu$ is cyclic in 
$\mathcal{A}_{\Lambda_\alpha}^{-\infty}$ for $\alpha>\alpha_0$.
\hfill $\Box$


\end{document}